\documentclass[11pt,english,a4paper]{smfart}
\usepackage[english]{babel}
\usepackage{amssymb,xspace}
\usepackage{amstext}
\usepackage{smfthm}
\theoremstyle{plain}
\usepackage{amsbsy,amssymb,amsfonts,latexsym}

\usepackage{enumerate}
\usepackage{graphics}

\title{Non-recurrence sets for weakly mixing linear dynamical systems}

\author{Sophie Grivaux}
\address{CNRS,
Laboratoire Paul Painlev\' e, UMR 8524, Universit\'e  Lille 1, Cit\' e Scientifique, 59655 Villeneuve d'Ascq
Cedex, France}
\email{grivaux@math.univ-lille1.fr}



\thanks{The author was partially supported
by ANR-Projet Blanc DYNOP}

\def\T{\ensuremath{\mathbb T}}
\def\R{\ensuremath{\mathbb R}}
\def\Z{\ensuremath{\mathbb Z}}

\def\C{\ensuremath{\mathbb C}}

\def\N{\ensuremath{\mathbb N}}

\DeclareMathOperator{\supp}{supp}

\newcommand{\sep}{separable}
\newcommand{\js}{Jamison sequence}

\newcommand{\ops}{operators}
\newcommand{\op}{operator}

\newcommand{\erg}{ergodic}
\newcommand{\eve}{eigenvector}
\newcommand{\eva}{eigenvalue}

\newcommand{\wrt}{with respect to}
\newcommand{\bs}{backward shift}
\newcommand{\proba}{probability}
\newcommand{\ga}{Gaussian}

\newcommand{\inv}{invariant}
\newcommand{\prob}{probability}
\newcommand{\mea}{measure}
\newcommand{\nd}{non-degenerate}
\newcommand{\mpt}{measure-preserving transformation}

\newcommand{\wmx}{weakly mixing}

\newcommand{\rg}{rigid}

\newcommand{\rs}{rigidity sequence}

\newcommand{\ds}{dynamical system}
\newcommand{\lds}{linear dynamical system}
\newcommand{\rec}{recurrence}
\newcommand{\nrec}{non-recurrence}
\newcommand{\rect}{recurrent}
\newcommand{\nrect}{non-recurrent}
\newcommand{\ppb}{partially power-bounded}

\newcommand{\pol}{polynomial}

\newcommand{\pss}[2]{\ensuremath{{\langle #1,#2\rangle}}}

\newtheorem{theorem}{Theorem}[section]

\newtheorem{lemma}[theorem]{Lemma}

\newtheorem{proposition}[theorem]{Proposition}

\newtheorem{corollary}[theorem]{Corollary}

{\theoremstyle{definition}}

{\theoremstyle{definition}\newtheorem{example}[theorem]{Example}}

\newtheorem{fact}[theorem]{Fact}

{\theoremstyle{definition}\newtheorem{definition}[theorem]{Definition}}

{\theoremstyle{definition}}

{\theoremstyle{definition}\newtheorem{remark}[theorem]{Remark}}

\newtheorem{question}[theorem]{Question}

{\theoremstyle{definition}\newtheorem*{FFC Criterion}{Frequent
Faber-hypercyclicity Criterion}}
\newtheorem*{Hypercyclicity Criterion}{Hypercyclicity Criterion}
{\theoremstyle{definition}\newtheorem*{GS Criterion}{Godefroy-Shapiro
Criterion}}
\def\piednote#1{\let\oldfn=\thefootnote\def\thefootnote{}\footnote{\noindent#1}%
\addtocounter{footnote}{-1}\def\thefootnote{\oldfn}}

\begin{document}

\begin{abstract}
We study non-recurrence sets for weakly mixing dynamical systems by using linear dynamical systems. These are systems consisting of a bounded linear operator acting on a separable complex Banach space $X$, which becomes a probability space when endowed with a non-degenerate Gaussian measure. We generalize some recent results of Bergelson, del Junco, Lema\'nczyk and Rosenblatt, and show in particular that sets $\{n_{k}\}$ such that $\frac{n_{k+1}}{n_{k}}\to +\infty$, or such that $n_{k}$ divides $n_{k+1}$ for each $k\ge 0$, are non-recurrence sets for weakly mixing linear dynamical systems. We also give examples, for each $r\ge 1$, of $r$-Bohr sets which are non-recurrence sets for some weakly mixing systems.
\end{abstract}
\maketitle

\section{Introduction}

The main topic of this paper is the study of recurrence and non-recurrence in the measure-theoretic framework, and for a particular class of dynamical systems, namely \wmx\ \ds s. Let $(n_{k})_{k\ge 0}$ be a strictly increasing sequence of positive integers. The set $\{n_{k} \textrm{ ; } k \ge 0\}$ is a \emph{\rec\ set}, or a \emph{Poincar\'e set}, if for any \ds\ $(X, \mathcal{B},m,T)$, where $T$ is a \mpt\ of a non-atomic probability space $(X,\mathcal{B},m)$, the following is true: for any $A\in \mathcal{B}$ with $m(A)>0$ there exists a $k\ge 0$ such that $m(T^{-n_{k}}A\cap A)>0$.
\par\smallskip
Recurrence is a central topic in \erg\ theory, and we refer the reader to one of the works \cite{Fu}, \cite{Fu2} or \cite{GW} and to the references therein for more information about it, as well as for applications to number theory and combinatorics. Standard examples of \rec\ sets are the set $\N$ of all integers (this is the classical Poincar\'e \rec\ theorem), the set of squares $\{k^{2}\textrm{ ; }k\ge 0\}$, and more generally any set $\{p(k)  ; k\ge 0\}$ where $p$ is a \pol\ taking integer values on integers such that $p(0)=0$. Also, any thick set, i.e. containing arbitrarily long blocks of integers, is a \rec\ set, as well as any set of the form $(E-E)\cap\N$, where $E$ is an infinite subset of $\N$. On the other hand, it is not difficult to exhibit \nrec\ sets: the set of odd integers is the easiest example, and rotations on the unit circle $\T=\{\lambda \in\C\textrm{ ; } |\lambda |=1\}$ provide a wealth of \nrec\ sets: if $\{n_{k}\}$ is any set for which there exists a $\lambda \in\T$ and a $\delta >0$ such that $|\lambda ^{n_{k}}-1|\ge \delta $ for all $k$, then $\{n_{k}\}$ is clearly not a \rec\ set.
\par\smallskip
In the rest of the paper, we will say that $\{n_{k}\}$ is a \emph{\nrec\ set for the \ds }\ $(X,\mathcal{B},m,T)$, or simply that
$\{n_{k}\}$ is \emph{\nrect\ for the \ds }\ $(X,\mathcal{B},m,T)$,
 if there exists a set $A\in \mathcal{B}$ with $m(A)>0$ such that $m(T^{-n_{k}}A\cap A)=0$ for each $k\ge 0$. Given a \nrec\ set $\{n_{k}\}$, it is interesting to try to construct  $(X,\mathcal{B},m,T)$ as above with additional properties: can we construct $T$ \erg? \wmx? The study of \nrec\ sets for \wmx\ \ds s was initiated by Bergelson, Del Junco, Lema\'nczyk and Rosenblatt in the recent paper \cite{BDLR}, where they give several examples of \wmx\ \nrect\ systems and study the relationship between \nrec\ and \rg ity. They prove here in particular that the generic transformation is both \wmx, \rg\ and \nrect, and that if the sequence $(n_{k})_{k\ge 0}$ grows sufficiently fast, namely if it satisfies the condition $$\sum_{k\ge 0}\frac{n_{k}}{n_{k+1}}<+\infty ,$$ then $\{n_{k}\}$ is a \nrec\ set for some \wmx\ \ds.
\par\smallskip
Our aim here is to continue this study of \nrec\ sets for \wmx\ systems. We first provide new examples of such sets by using a rather particular class of \ds s, namely linear \ds s. A \emph{\lds}\ consists of a bounded linear \op\ $T$ on a \sep\ complex infinite-dimensional Banach space $X$, and under certain conditions concerning, usually, the \eve s of $T$ associated to \eva s of modulus $1$, it is possible to construct a \prob\ \mea\ on $X$ \wrt\ which $T$ becomes a \mea-preserving \wmx\ transformation. More details about this will be given in the next section. We will also need to use some results about \emph{non-Jamison sequences}, which form a class of sequences connected to the study of partial power-boundedness of operators on \sep\ spaces, and which appear naturally as well in the study of \rg ity  sequences (see \cite{EG}). 
\par\smallskip
Here is our first result, which generalizes one of the  results of \cite{BDLR} mentioned above: 

\begin{theorem}\label{th1}
If $(n_{k})_{k\ge 0}$ is a sequence of integers such that
$\frac{n_{k+1}}{n_{k}}$ tends to infinity, the set $\{n_{k}\}$ is a \nrec\ set for some \wmx\ \lds. More generally, the same result holds if $(n_{k})_{k\ge 0}$ is a non-\js\ for which there exists a $\lambda _{0}\in\T$ such that $\inf_{k\ge 0}|\lambda _{0}^{n_{k}}-1|>0$ (i.e. if $\{n_{k}\}$ is a \nrec\ set for some rotation on the unit circle).
\end{theorem}

Our second result concerns  sequences $(n_{k})_{k\ge 0}$ such that $n_{k}$ divides $n_{k+1}$ for each $k\ge 1$.

\begin{theorem}\label{th2}
 Let $(n_{k})_{k\ge 0}$ be a sequence such that $n_{k}|n_{k+1}$ for each $k\ge 1$. Then  for any $p\in\Z$, the set $\{n_{k}-p \textrm{ ; } k\ge 0\}\cap\N$ is a \nrec\ set for some \wmx\ \lds. This applies in particular to the set $\{n_{k}\}$ itself.
\end{theorem}

The proof of Theorem \ref{th2} uses the notion of \rg ity, which was thoroughly explored in the context of \wmx\ systems in the paper \cite{BDLR}, and also in \cite{EG}. Part of the proof of Theorem \ref{th2} relies on the fact that if $(n_{k})_{k\ge 0}$ is a \rs\ in a certain strong sense, then $\{n_{k}-p \textrm{ ; } k\ge 0\}\cap\N$ is a \nrec\ set for some \wmx\ \ds, whatever the choice of $p\in\Z\setminus\{0\}$. It seems to be an open question whether, if $(n_{k})_{k\ge 0}$ is a \rs\ and $p\in\Z\setminus\{0\}$, the set  $\{n_{k}-p \textrm{ ; } k\ge 0\}\cap\N$ is always a \nrec\ set for some \wmx\ system.

\par\smallskip
A central open question in the paper \cite{BDLR} concerns lacunary sequences: if $(n_{k})_{k\ge 0}$ is a lacunary sequence, i.e. if there exists an $a>1$ such that $\frac{n_{k+1}}{n_{k}}>a$ for any $k$, then a result of Pollington \cite{P} and de Mathan \cite{dM} is that there exists an element $\lambda=e^{2i
\pi\theta} \in\T$ with $\theta$ irrational and a $\delta >0$ such that $|\lambda ^{n_{k}}-1|>\delta $ for all $k$. So $\{n_{k}\}$ is a \nrec\ set for some \erg\ \ds. Hence the following natural question:

\begin{question}\cite{BDLR}\label{q1}
 If $(n_{k})_{k\ge 0}$ is a lacunary sequence, does there always exist a \wmx\ \ds\ for which the set $\{n_{k}\}$ is not \rect?
\end{question}

Theorem \ref{th2} provides a positive answer to this question for sequences such that $n_{k}|n_{k+1}$ for each $k\ge 0$. Cutting and stacking constructions can also be used to
exhibit some \nrect\ lacunary sets (see Theorem \ref{prop2}, which generalizes the result of \cite{BDLR} that the set $\{\frac{3^{k+1}-1}{2}-1 \textrm{ ; } k\ge 0\}$ is \nrect\ for the Chacon transformation).
\par\smallskip

Our last result concerns sets which, for some fixed integer $r\ge 1$, are \rect\ (in the topological sense) for all products of $r$ rotations on the unit circle. In accordance with the terminology of \cite{K}, let us call such sets \emph{$r$-Bohr sets}. It is an open question, equivalent to an old combinatorial problem of Veech \cite{Ve} about syndetic sets and the Bohr topology on $\Z$, to know whether any set which is \rect\ for all finite products of rotations (such a set is called a \emph{Bohr set}) is necessarily topologically \rect\ for all \ds s. See \cite{Gl2}, \cite{GW}, \cite{BoGl} or \cite{K} for more about this question. It is shown in \cite{K} that there exists for each $r\ge 1$ sets which are $r$-Bohr but not $(r+1)$-Bohr. Another construction of $r$-Bohr sets which are not Bohr is given in \cite{GR}. One of the interests of the sets constructed in \cite{GR} is that they have density zero, contrary to the sets of \cite{K} which have positive density. Hence the sets of \cite{K}, which are \nrect\ for some product of $r+1$ rotations on $\T$, are recurrent for all \wmx\ \ds s. We show here, as a consequence of Theorem \ref{th1}, that the sets $\{n_{k}^{(r)}\}$ of \cite{GR} are $r$-Bohr, but are not recurrent for some \wmx\ dynamical systems:

\begin{theorem}\label{th3}
For any integer $r\ge 1$ there exist sets  $\{n_{k}^{(r)}\}$ of integers which are $r$-Bohr, but which are \nrec\ sets for some \wmx\ \lds s.
\end{theorem}

The paper is organized as follows: in Section $2$ we recall some results about \lds s and \js s which will be needed for the proofs of the theorems. In Sections $3$, $4$ and $5$ we present the proofs of Theorems \ref{th1}, \ref{th2} and \ref{th3} respectively. Section $6$ contains the proof of Theorem \ref{prop2} concerning a construction of \wmx\ \nrect\ systems by cutting and stacking.

\par\bigskip
\textbf{Acknowledgement:} I am grateful to Maria Roginskaya for stimulating discussions about some of the results of this paper.

\section{Weakly mixing linear dynamical systems and Jamison sequences}

In this section, which is expository, we briefly review some results concerning \lds s. They will be necesary for the proofs of the majority of the results in this paper. For a more detailed account, we refer the reader to the recent book \cite{BM}.

\subsection{Linear dynamical systems}
A \lds\ consists of a pair $(X,T)$, where $X$ is an infinite-dimensional complex \sep\ Banach space, and $T$ is a bounded linear \op\ on $X$. Under some conditions on 
$T$ and $X$, it is possible to construct a \nd\ \ga\ \mea\ $m$ on $X$ \wrt\ which $T$ defines a \mea-preserving \wmx\ transformation. Recall that if $\mathcal{B}$ denotes the $\sigma  $-algebra of Borel subsets of $X$, a Borel \proba\ \mea\ $m$ on $X$ is said to be (centered) \emph{\ga}\ if any element $x^{*}\in X^{*}$, considered as a complex-valued random variable on $(X,\mathcal{B},m)$, has (centered) \ga\ law: for any Borel subset $A$ of $\C$,
$$m(\{x\in X \textrm{ ; } \pss{x^{*}}{x} \in A\})=\frac{1}{2\pi \sigma _{x^{*}}^{2}}\int_{A}e^{-\frac{(u^{2}+v^{2})}{2\sigma  _{x^{*}}^{2}}}dudv$$
for some $\sigma  _{x^{*}}>0$. The \mea\ $m$ is said to be \emph{\nd}\ if its topological support is the whole space $X$, i.e if $m(U)>0$ for any non-empty open subset $U$ of $X$. Whenever $T$ is a \mpt\ of $(X, \mathcal{B},m)$, we will denote by $U_{T}$ the associated Koopman \op\ on $L^{2}(X, \mathcal{B},m)$: for $f\in L^{2}(X, \mathcal{B},m)$, $U_{T}f=f\circ T$.
\par\smallskip
The study of \ops\ on Banach spaces from the \erg\ point of view was initiated by Flytzanis in \cite{Fl}, and developed later on in the two papers \cite{BayGr1} and \cite{BayGr2}. Among other things, a necessary and sufficient condition was obtained, under which a bounded \op\ on a complex infinite-dimensional \sep\ Hilbert space $H$ admits a \nd\ \inv\ \ga\ \mea\ \wrt\ which it defines an \erg\ (or, equivalently here, \wmx)\ transformation of $X$. This condition involves \eve s of $T$ associated to \eva s of modulus $1$, which we call \emph{unimodular \eve s}. It states roughly that if $T\in \mathcal{B}(H)$ has a huge supply of unimodular \eve s (which happens in very many concrete situations), $T$ defines a \wmx\ \mpt\ of $H$ \wrt\ some \nd\ \ga\ \mea\ $m$. This condition is given in \cite{BayGr1} and \cite{BayGr2} in terms involving perfectly spanning sets of unimodular \eve s, but a simpler equivalent condition was recently obtained in the paper \cite{Gr}:

\begin{theorem}\label{th0}\cite{BayGr1}, \cite{Gr}
 Let $T$ be a bounded linear \op\ on a complex \sep\ infinite-dimensional Hilbert space $H$. The following assertions are equivalent:
 \begin{itemize}
  \item [(1)] for any countable subset $\Delta  $ of $\T$, the linear span of the eigenspaces $\ker(T-\lambda )$, $\lambda \in\T\setminus \Delta  $, is dense in $H$;
  
  \item [(2)] there exists a \nd\ \ga\ \mea\ $m$ on $H$ such that $T$ is a \wmx\ \mpt\ of $(H, \mathcal{B},m)$.
 \end{itemize}
\end{theorem}

We refer the reader to \cite{BayGr1}, \cite{BayGr2}, \cite{BoGE1} or \cite{BM} for instance for many examples of such \wmx\ \lds s, living not only on Hilbert spaces but also on other Banach or Fr\'echet spaces. Extensions of Theorem \ref{th0} to the Banach space setting were obtained in \cite{BayGr2} and \cite{BM}, culminating in the recent paper \cite{BM2} where it was shown that the implication $(1)\implies (2)$ remains true when $X$ is an arbitrary complex \sep\ Banach space.
\par\smallskip
Let us now present a particular class of \ops, which give in a very handy way examples of \ga\ \wmx\ systems on a Hilbert space associated to prescribed continuous \mea s on $\T$.

\subsection{Kalish-type operators}

These \ops\ were introduced by Kalish in \cite{Kal} in order to exhibit, for each closed subset $F$ of $\T$, examples of \ops\ on a Hilbert space whose spectrum and point spectrum coincide with $F$. The construction goes this way. On the space $L^{2}(\T)$ of square-integrable functions on $\T$, consider the \op\ $T=M-J$, where $M$ and $J$ are defined as follows: $M$ is the multiplication \op\ by the independent variable $\zeta $ on $L^{2}(\T)$, and $J$ is the integration operator: for any $f\in L^{2}(\T)$, $Mf:\zeta \mapsto \zeta f(\zeta )$ and $Jf:\zeta \mapsto \int_{(1,\zeta )}f$, where, if $\zeta =e^{i\theta }$ with $0\le \theta <2\pi$, $(1,\zeta )$ denotes the arc $\{e^{i\alpha } \textrm{ ; } 0\le \alpha \le \theta \}$, and $(\zeta ,1)$ denotes the arc $\{e^{i\alpha } \textrm{ ; } \theta \le \alpha \le 2\pi\}$. It is not difficult to check that for any $\lambda \in\T$ the characteristic function $\chi _{\lambda }$ of the arc $(\lambda ,1)$ satisfies $T\chi_{\lambda}=\lambda \chi_{\lambda}$, so that when $\chi_{\lambda}$ is non-zero, it
is an \eve\ of $T$ associated to the \eva\ $\lambda $.
\par\smallskip
Let now $\sigma  $ be a continuous \proba\ \mea\ on $\T$, and let $K$ denote its support. The closed linear span $H_{K}$ in $L^{2}(\T)$ of the \eve s  $\chi _{\lambda }$ of $T$, $\lambda \in K$, is $T$-invariant, and so $T$ induces an \op\ $T_{K}$ on $H_{K}$. One sees easily that the \eve\ field (for $T_{K}$) $E:K\rightarrow H_{K}$ which maps $\lambda $ to $\chi _{\lambda }$ is continuous, and that these \eve s $E(\lambda )$, $\lambda \in K$, span a dense subspace of $H_{K}$. Moreover the spectrum of $T_{K}$ coincides with $K$, and $T_{K}$ is invertible. It follows immediately from Theorem \ref{th0} that $T_{K}$ admits a \nd\ \ga\ \mea\ \wrt\ which $T_{K}$ becomes a \wmx\ \mpt\ of $H_{K}$, but we actually know more, and there is a ``canonical'' way to associate a suitable \ga\ \mea\ $m$ to the \eve\ field $E$ and the \mea\ $\sigma  $. One can construct a \mea\ $m$ such that for any $x,y\in H_{K}$,
$$\int_{H_{K}}\pss{x}{z}\overline{\pss{y}{z}}dm(z)=\int_{\T} \pss{x}{E(\lambda )}\overline{\pss{y}{E(\lambda )}}d\sigma  (\lambda ),$$ and the continuity of $E$ on $K$ combined with the fact that $\textrm{span}[E(\lambda ) \textrm{ ; } \lambda \in K]$ is dense in $H_{K}$ implies that $m$ is \nd. Also, it follows from the continuity of the \mea\ $\sigma  $ (see \cite{BayGr1} for details) that $T_{K}$ is \wmx\ \wrt\ this \mea\ $m$. Let us now define 
\begin{eqnarray*}
 V:& L^{2}(\T,\sigma  )&\rightarrow  L^{2}(\T,\sigma  )\\
 &f&\mapsto [\lambda \mapsto \lambda f(\lambda )]
\end{eqnarray*}
and
\begin{eqnarray*}
 A:& L^{2}(\T,\sigma  )&\rightarrow H_{K}\\
 &f&\mapsto \int_{\T}f(\lambda )E(\lambda )d\sigma  (\lambda ).
\end{eqnarray*}
One sees easily that $TA=AV$, and so we have for all $x,y\in H_{K}$ and all $n\ge 0$
\begin{eqnarray}\label{eq1}
\int_{H_{K}}\pss{x}{T_{K}^{n}z}\overline{\pss{y}{z}}dm(z) =\int_{\T}\lambda ^{n}\pss{x}{E(\lambda )}\overline{\pss{y}{E(\lambda )}}d\sigma  (\lambda ).
\end{eqnarray}

As explained in \cite{EG}, these Kalish-type \ops\  can be used to transfer properties of the \mea\ $\sigma  $ into \erg\  properties of the \op\ $T_{K}$. In particular, suppose that $\sigma  $ is such that $\hat{\sigma  }({n_{k}})\to 1$ as $k\to + \infty $ for some strictly increasing sequence $(n_{k})_{k\ge 0}$ of integers. It follows from (\ref{eq1}) that 
\begin{eqnarray*}
 \int_{H_{K}}\pss{x}{(T_{K}^{n_{k}}-I)z}\overline{\pss{y}{z}}dm(z) \to 0 
 \textrm{ as } {n_{k}\to + \infty } 
 \textrm{ for all }x,y \in H_{K}
\end{eqnarray*}
and also (see \cite{EG} for details) that $U_{T_{K}}^{n_{k}}\to I$ in the Weak Operator Topology. Thus $T_{K}$ is an example of a \wmx\ transformation which is \emph{\rg}\ 
\wrt\ $(n_{k})_{k\ge 0}$. The notion of \rg\ dynamical system was introduced by Furstenberg and Weiss in \cite{FuW} in their study of midly mixing \ds s (see also \cite{Fu}).

\begin{definition}
A \mpt\ $T$ of a non-atomic \proba\ space $(X  ,\mathcal{B}, \mu )$ is said to be \emph{\rg}\ if there exists a sequence $(n_{k})_{k\ge 0}$ of integers such that for any $A\in \mathcal{B}$, $\mu (T^{-n_{k}}A\triangle A)\to 0$ as $n_{k}\to +\infty $.
\end{definition}

Let us recall the following terminology of \cite{BDLR} and \cite{EG}: a sequence  $(n_{k})_{k\ge 0}$  is said to be a \emph{rigidity sequence} if there exists a \wmx\ \mpt\ which is \rg\ \wrt\  $(n_{k})_{k\ge 0}$. Rigidity sequences can be characterized in terms of Fourier coefficients of measures:

\begin{theorem}\label{th00}\cite{BDLR}, \cite{EG}
The following assertions are equivalent:
\begin{itemize}
 \item [(1)]  $(n_{k})_{k\ge 0}$  is a rigidity sequence;
 
 \item[(2)] there exists a continuous \proba\ \mea\ $\sigma  $ on $\T$ such that $\hat{\sigma  }({n_{k}})\to 1$ as $k\to + \infty $.
\end{itemize}
\end{theorem}

If (2) in Theorem \ref{th00} above is satisfied, the Kalish-type \op\ $T_{K}$ associated to $\sigma  $ is \wmx\ and \rg\ \wrt\ $(n_{k})_{k\ge 0}$.

\subsection{Partially power-bounded operators}

In the study of \rec\ we are undertaking,
an especially interesting class of \lds s is the class of \ppb\ \ops:

\begin{definition}
 If $(n_{k})_{k \ge 0}$ is a strictly increasing sequence of integers, and $T$ is a bounded linear \op\ on the Banach space
 $X$, $T$ is said to be \emph{\ppb}\ \wrt\ $(n_{k})_{k \ge 0}$  if $\sup_{k \ge 0}||T^{n_{k}}||$ is finite.
\end{definition}
 In the rest of the paper we will denote by $\sigma  _{p}(T)\cap\T$ the \emph{unimodular point spectrum} of the \op\ $T$, i.e. the set of \eva s of $T$ which are of modulus $1$.
\par\smallskip
 When the
 sequence $(n_{k})_{k \ge 0}$ is ``too rich'', such \ppb\ \ops\ \wrt\ $(n_{k})_{k\ge 0}$ can have only countably many \eva s on the unit circle. These sequences $(n_{k})_{k\ge 0}$ are called \emph{\js s}, because of an old result of Jamison which states that the sequence $(k)$ has this property.

\begin{definition}\cite{BaGr2}
A sequence $(n_{k})_{k \ge 0}$ of integers is said to be a \emph{Jamison sequence} when the following property holds true: if $X$ is any complex separable Banach space and $T$ is any bounded linear operator on $X$ such that $\sup_{k \ge 0}||T^{n_{k}}||$ is finite, then $\sigma  _{p}(T)\cap\T$ is at most countable.
\end{definition}

It can happen that for some sequences $(n_{k})_{k\ge 0}$ (the non-Jamison sequences), there exists an \op\ on
some complex \sep\ Banach space $X$ which is 
\ppb\ \wrt\ $(n_{k})_{k\ge 0}$ and which has uncountably many \eva s on the unit circle. In this case, there are hopes that $T$ could define a \wmx\ transformation of $(X,\mathcal{B},m)$ for some suitable \ga\ \mea\ $m$ on $X$.
 \par\smallskip
 The study of \ppb\ \ops\ started with a paper of Ransford \cite{Ran}, and the first example of a \ppb\ \op\ with uncountable unimoduar point spectrum was given by Ransford and Roginskaya in \cite{RR}. Jamison and non-Jamison sequences were then further investigated in \cite{BaGr} and \cite{BaGr2}, and a complete characterization of \js s was obtained in \cite{BaGr2}:
 
 \begin{theorem}\label{th000}\cite{BaGr2}
  Let $(n_{k})_{k\ge 0}$ be a strictly increasing sequence of integers with $n_{0}=1$. The following assertions are equivalent:
  \begin{itemize}
   \item [(1)] $(n_{k})_{k\ge 0}$ is a \js, i.e. for any \sep\ Banach space $X$ and any bounded \op\ $T\in \mathcal{B}(X)$, $\sup_{k\ge 0}||T^{n_{k}}||<+\infty $ implies that $\sigma  _{p}(T)\cap\T$ is at most countable;
   
   \item [(2)] there exists an $\varepsilon >0$ such that for any 
  $ \lambda \in\T\setminus\{1\}$, $\sup_{k\ge 0}|\lambda ^{n_{k}}-1|\ge \varepsilon $.
  \end{itemize}
 \end{theorem}

Several examples of non-Jamison sequences were given in \cite{BaGr} and \cite{BaGr2}. Among these are the sequences $(n_{k})_{k\ge 0}$ such that $\frac{n_{k+1}}{n_{k}}\to +\infty $, the sequences  $(n_{k})_{k\ge 0}$ such that $n_{k}|n_{k+1}$ for each $k$ and $\limsup_{k\to+\infty }\frac{n_{k+1}}{n_{k}}=+\infty $, etc. In all these cases (and in general, when $(n_{k})_{k\ge 0}$ is not a \js), there exists a \sep\ $X$ and $T\in \mathcal{B}(X)$ such that $\sup_{k\ge 0} ||T^{n_{k}}||<+\infty $ and $\sigma  _{p}(T)\cap\T$ is uncountable. It was recently shown in \cite{EG} that one could always take $X$  to be a Hilbert space, and, moreover, that the \op\ $T$ could be constructed in such a way that it defines  a \wmx\ \mpt\ of $(H, \mathcal{B},m)$ for some \nd\ \ga\ \mea\ $m$. This can be summarized as follows:

\begin{theorem}\label{th0000}\cite{EG}
  Let $(n_{k})_{k\ge 0}$ be a strictly increasing sequence of integers such that $n_{0}=1$. The following assertions are equivalent:
  \begin{itemize}
   \item [(1)] $(n_{k})_{k\ge 0}$ is not  a Hilbertian \js, i.e. there exists a bounded linear \op\ $T$ on a complex \sep\ Hilbert space $H$ such that $\sup_{k\ge 0}||T^{n_{k}}||<+\infty $ and $T$ defines a \wmx\ \mpt\ of $(H, \mathcal{B},m)$ for some \nd\ \ga\ \mea\ $m$ (in particular $\sigma  _{p}(T)\cap\T$ is uncountable);
   
   \item [(2)] for any $\varepsilon >0$ there exists a
  $ \lambda \in\T\setminus\{1\}$ such that $\sup_{k\ge 0}|\lambda ^{n_{k}}-1|< \varepsilon $.
  \end{itemize}
\end{theorem}

As can easily be guessed, the partial power-boundedness condition gives us a non-recurrence property, and its interest here is clear. We now have all our tools in hand for the proofs of Theorems \ref{th1}, \ref{th2} and \ref{th3}.

\section{Proof of Theorem \ref{th1}}

We are going to prove that if $(n_{k})_{k\ge 0}$  is a non-\js\ with $n_{0}=1$ and if $\lambda _{0}\in\T$ is such that $\inf_{k\ge 0}|\lambda _{0}^{n_{k}}-1| >0$, then there exists a bounded linear \op\ on a \sep\ complex Hilbert space $H$ which is \wmx\ and \nrect\ \wrt\  the set $\{n_{k}\}$. In spirit this statement follows from Theorem \ref{th0000}, 
but one needs to use some fine aspects of the construction in \cite{EG} in order to obtain it.
\par\smallskip
 The proof of \cite{EG} yields for any fixed number $\delta >0$ an \op\ $T\in \mathcal{B}(H)$ such that $\sup_{k\ge 0}||T^{n_{k}}||<1+\delta $ and $T$ is a \wmx\ transformation of $(H, \mathcal{B},m)$ for some \nd\ \ga\ \mea\ $m$. The \op\ is constructed as a sum 
$T=D+B$
of a diagonal \op\ $D$ and a weighted \bs\ $B$ on the space $\ell_{2}(\N)$ where:

$\bullet$ $D=\textrm{diag}(\lambda _{n} \textrm{ ; } n\ge 1)$ is diagonal \wrt\ the canonical basis $(e_{n})_{n\ge 1}$ of $\ell_{2}(\N)$, and the $\lambda _{n}$'s are distinct unimodular coefficients;

$\bullet$ $B$ is defined as a weighted \bs\ with $Be_{1}=0$ and $Be_{n}=\alpha _{n-1}e_{n-1}$ for $n\ge 2$, and $(\alpha _{n})_{n\ge 1}$ is a very quickly decreasing sequence of positive numbers.

 The construction uses an auxiliary function $j:\{2, +\infty \}\rightarrow
\{1,+\infty \}$ which has the properties that $j(n)<n$ for each $n$, $j(2)=1$, and $j$
takes each value $r\ge 1$ infinitely often. The elements $\lambda _{n}$, $n\ge 1$, are constructed by induction on $n$, and  the crucial tool for this construction is a distance $d_{(n_{k})}$ on the unit circle which is associated to the sequence $(n_{k})_{k\ge 0}$ in the following way: for $\lambda ,\mu \in\T$, $$d_{(n_{k})}(\lambda ,\mu )=
\sup_{k\ge 0}|\lambda ^{n_{k}}-\mu ^{n_{k}}|.$$ Since $n_{0}=1$, $d_{(n_{k})}$ is indeed a distance on $\T$. Now, the assumption that $(n_{k})_{k\ge 0}$ is not a \js\ implies that there exists an uncountable subset $K$ of $\T$ containing the point $1$ such that $(K, d_{(n_{k})})$ is perfect. Then the construction of the diagonal coefficients $\lambda_{n}$ goes as follows: one first chooses $\lambda _{1}\in K$. Then if $\delta $ is a positive number, one proves that there exists a sequence $(\varepsilon _{n})$
of positive numbers going to zero very quickly such that if for each $n \ge 2$ we have $d_{(n_{k})}(\lambda _{n},\lambda _{j(n)})<\varepsilon _{n}$, then for a suitable choice of the sequence of weights $(\alpha _{n})$ the  \op\ $T=D+B$ is such that $\sup_{k\ge 0}||T^{n_{k}}-D^{n_{k}}||<\delta $, and is \wmx. Observe also that $T$ is invertible if the weights $\alpha _{n}$ are sufficiently small.
\par\smallskip
We now claim that we can ensure that $\sup_{k}||T^{n_{k}}-I||<2\delta $. Indeed we have for each $k\ge 1$ that
$||T^{n_{k}}-I||\le ||T^{n_{k}}-D^{n_{k}}||+||D^{n_{k}}-I||$, and $||D^{n_{k}}-I||=\sup_{n\ge 1}|\lambda _{n}^{n_{k}}-1|$.
Hence $\sup_{k}||D^{n_{k}}-I||=\sup_{k}\sup_{n}|\lambda _{n}^{n_{k}}-1|=\sup_{n}d_{(n_{k})}(\lambda _{n},1)$.
Now $$d_{(n_{k})}(\lambda _{n},1)\le d_{(n_{k})}(\lambda _{n},\lambda _{j(n)})+d_{(n_{k})}(\lambda _{j(n)},\lambda _{j\circ j (n)})+\dots + d_{(n_{k})}(\lambda _{j^{[p-1]}(n)},\lambda _{j^{[p]}(n)}),$$ where $j^{[k]}$ denotes the composition of $j$ with itself $k$ times and $p$ is the unique integer such that $j^{[p-1]}(n)>1$ and $j^{[p]}(n)=1$ (recall that $j(2)=1$). Hence
$$d_{(n_{k})}(\lambda _{n},1)\le \varepsilon _{n}+\varepsilon _{j(n)}+\ldots + \varepsilon _{j^{[p-1]}(n)}\le \sum_{n\ge 2}\varepsilon _{n}.$$
If we take the $\varepsilon _{n}$'s sufficiently small, we can ensure that $\sum_{n\ge 2}\varepsilon _{n}$ is less than any prescribed positive number, for instance less than $\delta $. So we get that $\sup_{k}||T^{n_{k}}-I||<2\delta $.
\par\smallskip
Going back to our initial assumption, we know that there exists $\lambda _{0}\in\T$ such that $\inf_{k\ge 0}|\lambda _{0}^{n_{k}}-1|=\delta  >0$. Let
 $T\in \mathcal{B}(H)$ be as above with $\sup_{k\ge 0}||T^{n_{k}}-I||<\frac{\delta }{2}$,
and consider the \op\ $S=\lambda _{0}T$. 
By construction, $T$ is such that for any countable subset $\Delta  $ of $\T$, the linear space spanned by the kernels
$\ker(T-\lambda)$, $ \lambda \in\T\setminus \Delta  $, is dense in $H$. It is obvious that $\lambda _{0}T$ satisfies the same property, and so by Theorem \ref{th0} $S$ is \wmx\ \wrt\ some \nd\ \ga\ \mea\ $m$. For $\gamma \in (0,1)$, let $U_{\gamma }=B({e_{1},\gamma })$ be the open ball of $H$ centered at the vector $e_{1}$ and of radius $\gamma $. For $u\in U_{\gamma }$ we have for all $k\ge 1$
\begin{eqnarray*}
||S^{n_{k}}u-u||&=&||\lambda _{0}^{n_{k}}T^{n_{k}}u-u||=||\lambda _{0}^{n_{k}}(T^{n_{k}}u-u)+(\lambda _{0}^{n_{k}}-1)u||\\
&\ge&|\lambda _{0}^{n_{k}}-1|\,||u||-\frac{\delta }{2} ||u||\\
&\ge& \delta (1-\gamma )-\frac{\delta }{2}(1+\gamma ) =\frac{\delta }{2}-\frac{3\delta}{2}\gamma>2\gamma 
\end{eqnarray*}
if $\gamma $ is sufficiently small. Hence $S^{n_{k}}U_{\gamma }\cap U_{\gamma }=\varnothing$ if $\gamma $ is sufficiently small, and so $\{n_{k}\}$ is a \nrec\ set for the \op\ $S$.
\par\smallskip
If $(n_{k})_{k\ge 0}$ is a lacunary sequence, a result of de Mathan \cite{dM} and Pollington \cite{P} (see also 
\cite{K}) states that there exists indeed a $\lambda _{0}\in\T$ such that $\inf_{k\ge 0}|\lambda _{0}^{n_{k}}-1| >0$. So any lacunary non-\js\ is a \nrec\ set for some \wmx\ \lds. In particular, thanks to a result which we recalled in Section $2$, this is the case if $\frac{n_{k+1}}{n_{k}}$ tends to infinity as $k$ tends to infinity. Theorem \ref{th1} is proved.

\begin{remark}\label{rem1}
 It is clearly not true that any non-\js\ yields a \nrec\ set for some \mea-preserving \ds. Indeed, if $(p_{j})_{j\ge 1}$ is a sequence of integers growing very fast, it is not difficult to see (thanks to Theorem \ref{th00}) that the sequence $(n_{k})_{k\ge 0}$  defined by  $$\{n_{k}\}=\bigcup_{k\ge 1}\{p_{j},2p_{j},\ldots, jp_{j}\}$$ is a non-\js. But $\{n_{k}\}$ is a \rec\ set, and of course there exists no element $\lambda _{0}\in\T$ such that $\inf_{k}|\lambda _{0}^{n_{k}}-1|>0$.
\end{remark}

\begin{remark}\label{rem2}
 In view of the results of \cite{BDLR}, one might also wonder whether it is true that whenever $(n_{k})_{k\ge 0}$ is a non-\js, the set $\{n_{k}-1\}$ is a \nrec\ set for some \wmx\ \lds. This is not the case: consider, as in Remark \ref{rem1} above, the set $$\{n_{k}\}=\bigcup_{k\ge 1}\{p_{j}+1,2p_{j}+1,\ldots, jp_{j}+1\},$$ where  $(p_{j})_{j\ge 1}$ is a rapidly growing sequence of integers. Then $(n_{k})$ is a non-\js, but $\{n_{k}-1\}$ is a recurrence set. Observe that $\{n_{k}\}$ itself is a \nrec\ set for some \wmx\ \lds. One can show, using the same kind of arguments as in the proof of Theorem \ref{th1}, that if $(n_{k})_{k\ge 0}$ is a non-\js\ and if there exists for each $\varepsilon >0$ a $\lambda \in\T$ such that $\sup_{k}|\lambda ^{n_{k}}-1|<\varepsilon $ and $|\lambda ^{n_{k}}-1|\to 0$, then for every $p\in\Z\setminus\{0\}$ the set $\{n_{k}-p\}\cap\N$ is a \nrec\ set for some \wmx\ \lds.
\end{remark}

\section{Proof of Theorem \ref{th2}}

Let $(n_{k})_{k\ge 0}$ be a sequence of integers such that $n_k|n_{k+1} $ for each $k \ge 0$. It is proved in \cite{BDLR} and \cite{EG} that
$(n_{k})_{k\ge 0}$ is a \rs\ for some \wmx\ \ds. As recalled in Section $2$, this is equivalent to saying that there exists a 
continuous \proba\ \mea\ $\sigma$ on $\T$ such that $\hat{\sigma}(n_k)\to 1$ as $n_k\to +\infty$. When $n_k|n_{k+1} $ for each $k \ge 0$ such 
\mea s do exist (they can be constructed as infinite convolutions of some suitable discrete \mea s), and one can even control the rate at
which the Fourier coefficients $\hat{\sigma}(n_k)$ tend to $ 1$. The following fact, due to J.-P. Kahane and contained in \cite{EG}, is the crux of the proof
of Theorem \ref{th2}.

\begin{lemma}\label{lem1}\cite{EG}
 Let $(n_{k})_{k\ge 0}$ be a sequence of integers such that $n_k|n_{k+1} $ for each $k \ge 0$. For any decreasing sequence
 $(a_k)_{k\ge 0}$ of positive numbers going to zero as $k$ goes to infinity and such that $\sum_{k\ge 0}a_k=+\infty$, there exists a 
continuous \proba\ \mea\ $\sigma$ on $\T$ such that $$|\hat{\sigma}(n_k)- 1|\le a_k \quad \textrm{ for all } k\ge 0.$$
\end{lemma}

We will also need the following elementary fact:

\begin{fact}\label{fact1}
 There exist two sequences $(a_k)_{k\ge 0}$ and $(b_k)_{k\ge 0}$ of positive real numbers and a partition of $\N$ in two infinite sets
 $A$ 
and $B$ such that the following properties hold true:
\begin{itemize}
 \item [(1)] the sequences $(a_k)_{k\ge 0}$ and $(b_k)_{k\ge 0}$ decrease to zero, and $0<a_k, b_k<1$ for each $k\ge 0$;

\item [(2)] the series $\sum_{k\ge 0}a_k$ and $\sum_{k\ge 0}b_k$ are divergent;

\item [(3)] the series $\sum_{k\in A}a_k^{\frac{1}{3}}$ and $\sum_{k\in B}b_k^{\frac{1}{3}}$ are convergent.
\end{itemize}
\end{fact}

\begin{proof}[Proof of Fact \ref{fact1}]
 We construct by induction a very quickly increasing sequence
 of integers  $(p_n)_{n\ge 0}$ and blocks $(a_k)_{k=p_{n-1}+1,\ldots, p_n}$ and $(b_k)_{k=p_{n-1}+1,\ldots, p_n}$ in such a way that
$$\sum_{k=p_{n-1}+1}^{p_n}a_k\ge \frac{1}{2}\quad \textrm{ and } \quad\sum_{k=p_{n-1}+1}^{p_n}b_k^{\frac{1}{3}}\le 2^{-n}\quad 
\textrm{ for each odd integer } n $$
and
$$\sum_{k=p_{n-1}+1}^{p_n}a_k^{\frac{1}{3}}\le 2^{-n}\quad \textrm{ and } \quad\sum_{k=p_{n-1}+1}^{p_n}b_k\ge \frac{1}{2}\quad 
\textrm{ for each even integer } n $$
and we set $$A=\bigcup_{n \textrm{ even}} \{p_{n-1}+1, \ldots, p_n\}\quad \textrm{ and } \quad B=\bigcup_{n \textrm{ odd}}
 \{p_{n-1}+1, \ldots, p_n\}.$$ We start with $p_0=0$, $p_1=1$, $a_1=\frac{1}{2}$ and $b_{1}=\frac{1}{8}$. Let now $r\ge 1$, and suppose that
$p_1,\ldots, p_{2r-1}$ have been constructed,  as well as $a_k$ and $b_k$ for $k\le p_{2r-1}$. We first choose $p_{2r}$ so large that
 $(p_{2r}-p_{2r-1})b_{p_{2r-1}}\ge \frac{1}{2}$, then set $b_k=b_{p_{2r-1}}$ for $p_{2r-1}+1\le k \le p_{2r}$. Then we choose,
 for $p_{2r-1}+1\le k \le p_{2r}$, $a_k$ positive and so small that $\sum_{k=p_{2r-1}+1}^{p_{2r}}a_k^{\frac{1}{3}}\le 2^{-(2r)}$, with $a_{k+1}\le a_{k}$ for each $k=p_{2r-1}+1, \ldots, p_{2r}-1$.
 The construction of $p_{2r+1}$ and of the numbers $a_k$ and $b_k$ for $p_{2r}+1\le k \le p_{2r+1}$ is exactly the same, permuting the
 roles of $a_k$ and $b_k$.
\end{proof}

We will have to consider separately two cases in the proof of Theorem \ref{th2}: the case where $p\in\Z$ is non-zero and the case where $p=0$.
\subsection{Proof of Theorem \ref{th2} when $p\in\Z\setminus \{0\}$}
 Let $(a_{k})_{k\ge 0}$, $(b_{k})_{k\ge 0}$, $A$ and $B$ be given by Fact \ref{fact1}. Since $n_{k}|n_{k+1}$ for each $k\ge 0$ and the two series $\sum a_{k}$ and $\sum b_{k}$ are divergent, there exist by Lemma \ref{lem1} two continuous \proba\ \mea s  $\sigma  $ and $\tau $ on $\T$
 such that $|\hat{\sigma  }(n_{k})-1|\le a_{k}$ and $|\hat{\tau  }(n_{k})-1|\le b_{k}$ for all $k\ge 1$. Also, the two series $\sum_{k\in A} |\hat{\sigma  }(n_{k})-1|^{\frac{1}{3}}$
and $\sum_{k\in B} |\hat{\tau  }(n_{k})-1|^{\frac{1}{3}}$ are convergent.
\par\smallskip
Let $T_{K}$ be the Kalish-type \op\ on $H_{K}$ associated to the compact set $K=supp(\sigma  )$, and $T_{L}$ the \op\ on $H_{L}$ associated to $L=supp(\tau )$ (where the symbol $\supp$ denotes the support of the \mea). Let $m_{\sigma  }$ and $m_{\tau}$ be the two \ga\ \mea s, on $H_{K}$ and $H_{L}$ respectively, defined in Section 2.2: for any $x,y\in H_{K}$
$$\int_{H_{K}}\pss{x}{z}\overline{\pss{y}{z}}dm_{\sigma  }(z)=\int_{\T}\pss{x}{E(\lambda )}\overline{\pss{y}{E(\lambda )}}d\sigma  (\lambda )$$ and for any $u,v\in H_{L}$
$$\int_{H_{L}}\pss{u}{w}\overline{\pss{v}{w}}dm_{\tau  }(w)=\int_{\T}\pss{u}{E(\lambda )}\overline{\pss{v}{E(\lambda )}}d\tau  (\lambda ),$$ where $E(\lambda )=\chi_{\lambda }$ for each $\lambda \in\T$. The \mea\ $m_{\sigma  }$ (resp. $m_{\tau}$) is \nd, \inv\ \wrt\ $T_{K}$ (resp. $T_{L}$), and $T_{K}$ (resp. $T_{L}$) is \wmx\ on $(H_{K}, \mathcal{B}_{K}, m_{\sigma  })$ (resp. on $(H_{L}, \mathcal{B}_{L}, m_{\tau  })$).
Our aim is now to show the following statement:

\begin{proposition}\label{prop1}
For any $p\in\Z\setminus\{0\}$, the set $\{n_{k}-p \textrm{ ; }k\ge 0\}\cap\N$
is a \nrec\ set for the \wmx\ system $T_{K}\times T_{L}$ on $(H_{K}\times H_{L}, \mathcal{B}_{K}\times \mathcal{B}_{L}, m_{\sigma  }\times m_{\tau})$.
\end{proposition}

\begin{proof}
Let us first observe that for any $x\in H_{K}$ and any $k\ge 1$,
\begin{eqnarray*}
\int_{H_{K}}|\pss{x}{(T_{K}^{n_{k}}-I)z}|^{2} dm_{\sigma  }(z)&=&\int_{\T} |\lambda ^{n_{k}}-1|^{2}|\pss{x}{E(\lambda )}|^{2}d\sigma  (\lambda )\\ 
&\le&2\pi\, ||x||^{2}\,\int_{\T}  |\lambda ^{n_{k}}-1|^{2}d\sigma  (\lambda )\\
&\le& 4\pi\, ||x||^{2}\, \Re e (1-\hat{\sigma  }(n_{k}))\le 4\pi \,||x||^{2}\, a_{k},
\end{eqnarray*}
and in the same way
\begin{eqnarray*}
\int_{H_{L}}|\pss{u}{(T_{L}^{n_{k}}-I)w}|^{2} dm_{\tau }(w)&=&\int_{\T} |\lambda ^{n_{k}}-1|^{2}|\pss{u}{E(\lambda )}|^{2}d\tau  (\lambda )\\ 
&\le&2\pi\, ||u||^{2}\,\int_{\T}  |\lambda ^{n_{k}}-1|^{2}d\tau  (\lambda )\\
&\le& 4\pi\, ||u||^{2}\, \Re e (1-\hat{\tau  }(n_{k}))\le 4\pi \,||u||^{2}\, b_{k}
\end{eqnarray*}
for all $u\in H_{L}$ and all $k\ge 1$.
For $x\in H_{K}$, let us denote by $f_{x}$ the function in $L^{2}(H_{K}, \mathcal{B}_{K},m_{\sigma  })$ defined by $f_{x}(z)=\pss{x}{z}$, and, for $u\in H_{L}$, by $g_{u}$ the function in $L^{2}(H_{L}, \mathcal{B}_{L},m_{\tau  })$ defined by 
$g_{u}(w)=\pss{u}{w}$. We have by the inequalities above
\begin{eqnarray}\label{eq2}
 ||f_{x}\circ T^{n_{k}}_{K}-f_{x}||^{2}\le  4\pi \,||x||^{2}\, a_{k}\quad \textrm{and}\quad 
 ||g_{u}\circ T^{n_{k}}_{L}-g_{u}||^{2}\le  4\pi \,||u||^{2}\, b_{k},
\end{eqnarray}
 and so
we obtain that for all $x\in H_{K}$ and all $u\in H_{L}$,
\begin{eqnarray*}
\sum_{k\in A} ||f_{x}\circ T^{n_{k}}_{K}-f_{x}||^{\frac{2}{3}}_{L^{2}(H_{K}, m_{\sigma  })}<+\infty 
\quad \textrm{and}\quad 
\sum_{k\in B} ||g_{u}\circ T^{n_{k}}_{L}-g_{u}||^{\frac{2}{3}}_{L^{2}(H_{L}, m_{\tau })}<+\infty.
\end{eqnarray*}
We will now need a modification of an argument of \cite{BDLR}:

\begin{lemma}\label{lem3}
Let $x_{0}$ be a non-zero vector of $H_{K}$, and let $a, b, c$ and $d$ be four real numbers with $a<b$ and $c<d$. Denote by $R_{a,b}^{c,d}$ the open rectangle of $\C$ defined by  $R_{a,b}^{c,d}=]a,b[+i]c,d[$. Then the set $$U_{a,b}^{c,d}=\{x\in H_{K}
\textrm{ ; } \pss{x_{0}}{x}\in R_{a,b}^{c,d}\}$$ is such that $m_{\sigma}(U_{a,b}^{c,d})>0$ and $\sum_{k\in A}m_{\sigma}(U_{a,b}^{c,d}\triangle T_{K}^{n_{k}}U_{a,b}^{c,d})$ is finite.
\end{lemma}

\begin{proof}[Proof of Lemma \ref{lem3}]
For simplicity, let us write $U_{a,b}^{c,d}=U$. Since $U$ is a non-empty open subset of $H_{K}$ and $m_{\sigma}$ is \nd, $m_{\sigma}(U)>0$. Our aim is now to show that there exists a positive constant $C$ such that for all $k$,
$m_{\sigma}(U\setminus T_{K}^{-n_{k}}U)\le Ca_{k}^{\frac{1}{3}}$.
Fix $\varepsilon >0$, and suppose that $x$ belongs to $U\setminus T_{K}^{-n_{k}}U$: 
we have $a<\Re e \pss{x_{0}}{x}<b$ and $c<\Im m \pss{x_{0}}{x}<d$, and either
 $$\Re e \pss{x_{0}}{T_{K}^{n_{k}}x}\le a-\varepsilon \quad \textrm{ or } \quad a-
\varepsilon<\Re e \pss{x_{0}}{T_{K}^{n_{k}}x}\le a$$
or $$ b\le \Re e \pss{x_{0}}{T_{K}^{n_{k}}x}<b+\varepsilon
\quad \textrm{ or } \quad b+\varepsilon\le \Re e \pss{x_{0}}{T_{K}^{n_{k}}x}$$
or $$\Im m \pss{x_{0}}{T_{K}^{n_{k}}x}\le c-\varepsilon \quad \textrm{ or } \quad
c-\varepsilon<\Im m \pss{x_{0}}{T_{K}^{n_{k}}x}\le c$$
or
$$d\le \Im m \pss{x_{0}}{T_{K}^{n_{k}}x}<d+\varepsilon
\quad \textrm{ or } \quad
d+\varepsilon\le \Im m \pss{x_{0}}{T_{K}^{n_{k}}x}.$$
 First observe that
$$\int_{\{x\in U \textrm{ ; } \Re e \pss{x_{0}}{T_{K}^{n_{k}}x}\le a-\varepsilon \}}
|\Re e \pss{x_{0}}{T_{K}^{n_{k}}x}-\Re e \pss{x_{0}}{x}|^{2}dm_{\sigma  }(x)$$ is bigger than $$
\varepsilon ^{2}\, m_{\sigma  }(\{x\in U \textrm{ ; } \Re e \pss{x_{0}}{T_{K}^{n_{k}}x}\le a-\varepsilon \}).$$ Thus we obtain from (\ref{eq2}) that for all $k\ge 1$,                                                                                                                                                                                                                                                                                               $$m_{\sigma  }(\{x\in U \textrm{ ; } \Re e \pss{x_{0}}{T_{K}^{n_{k}}x}\le a-\varepsilon \})\le \frac{4\pi\,||x_{0}||^{2}\,a_{k}}{\varepsilon ^{2}}\cdot$$ 
In the same way we get that the three quantities
$m_{\sigma  }(\{x\in U \textrm{ ; } \Re e \pss{x_{0}}{T_{K}^{n_{k}}x}\ge b-\varepsilon \})$,
$m_{\sigma  }(\{x\in U \textrm{ ; } \Im m \pss{x_{0}}{T_{K}^{n_{k}}x}\le c-\varepsilon \})$
and $m_{\sigma  }(\{x\in U \textrm{ ; } \Im m \pss{x_{0}}{T_{K}^{n_{k}}x}\ge d-\varepsilon \})$ are all smaller than $ {4\pi ||x_{0}||^{2}a_{k}}{\varepsilon ^{-2}}$.
 Now since $f_{x_{0}}$ has complex \ga\ distribution, $\Re e f_{x_{0}}$ and $\Im m f_{x_{0}}$ are two real variables having the same Gaussian distribution, and so   there exists a positive constant $M$ such that for any $\varepsilon \in (0,1)$  and any interval $I$ of $\R$ of length $\varepsilon$,                                                                                                                                                                                                                                                                                                                                                                                                                                                                                                                                                                                                                                                                                                                                                                                                                                                                                                                                                                                                                            $$m_{\sigma  }(\{x\in H_{K} \textrm{ ; } \Re e \pss{x_{0}}{x}\in I\})\le M\,\varepsilon $$
and
$$m_{\sigma  }(\{x\in H_{K} \textrm{ ; }  \Im m \pss{x_{0}}{x}\in I\})\le M\,\varepsilon .$$
 Since the \mea\ $\sigma  $ is $T_{K}$-\inv, 
$$m_{\sigma  }(\{x\in H_{K} \textrm{ ; }a-\varepsilon < \Re e \pss{x_{0}}{T_{K}^{n_{k}}x}\le a\})=
m_{\sigma  }(\{x\in H_{K} \textrm{ ; } a-\varepsilon < \Re e \pss{x_{0}}{x}\le a\})\le M\varepsilon$$ for all $k\ge 1$, and we have similar inequalities with the sets involving $b, c$ and $d$. 
We deduce from all this that for all $k\ge 1$ and all $\varepsilon >0$
$$m_{\sigma}({\{x\in U \textrm{ ; } \Re e \pss{x_{0}}{T_{K}^{n_{k}}x}\le a \}})\le \frac{4\pi ||x_{0}||^{2}a_{k}}{\varepsilon^{2}}+M\varepsilon,$$
and so that
\begin{eqnarray*}
 m_{\sigma  }(U\setminus T_{K}^{n_{k}}U)&\le & 4\left(\frac{4\pi ||x_{0}||^{2}a_{k}}{\varepsilon^{2}}+M\varepsilon\right)\le C\left(\frac{a_{k}}{\varepsilon ^{2}}+\varepsilon\right)
\end{eqnarray*}
for some numerical constant $C$. This being true for all $\varepsilon \in (0,1)$, we obtain by taking $\varepsilon= a_{k}^{\frac{1}{3}}$ that for all $k\ge 1$ $$m_{\sigma  }(U\setminus T_{K}^{-n_{k}}U)\le 2C\, a_{k}^{\frac{1}{3}}.$$ This implies that 
$m_{\sigma  }(U\cap T_{K}^{-n_{k}}U)\ge m_{\sigma}(U)-2C\, a_{k}^{\frac{1}{3}}$, and since
$m_{\sigma  }(T_{K}^{-n_{k}}U\triangle U)=2m_{\sigma}(U)-2m_{\sigma  }(U\cap T_{K}^{-n_{k}}U)$, we eventually obtain that $m_{\sigma  }(U\triangle T_{K}^{-n_{k}}U)=m_{\sigma  }(U\triangle T_{K}^{n_{k}}U)\le 4C\, a_{k}^{\frac{1}{3}}$.
Since the series $\sum_{k\in A}a_{k}^{\frac{1}{3}}$ is convergent, this yields that 
$$\sum_{k\in A}m_{\sigma  }(T_{K}^{n_{k}}U\triangle U)<+\infty .$$
Lemma \ref{lem3} is proved.
\end{proof}

Going back to the proof of Proposition \ref{prop1}, fix $p\in\Z\setminus\{0\}$ and let $U=U_{a,b}^{c,d}$ be one of the sets defined in Lemma \ref{lem3}. The argument
 runs now exactly as in \cite{BDLR}.

 We know that $T_{K}$ is invertible, \wmx, and so that all its powers are \erg. It follows that $T_{K}^{p}U$ is not contained in $U$ up to a set of \mea\ zero. Hence there exists a Borel subset $C_{0}$ of $U$, with $m_{\sigma  }(C_{0})>0$, such that $m_{\sigma  }(U\cap T_{K}^{p}C_{0})=0$. Let us set now, for a certain integer $\kappa $ to be fixed later on,
$$C=T_{K}^{p}C_{0}\setminus \bigcup_{k\ge \kappa, \, k\in A }T_{K}^{-(n_{k}-p)}(T_{K}^{n_{k}}U\triangle U).$$ We have
 $$m_{\sigma}(C)\ge m_{\sigma}(C_{0})-\sum_{k\ge \kappa, k\in A}m_{\sigma}(T_{K}^{n_{k}}U\triangle U).$$ Since the series on the right-hand side is convergent by Lemma \ref{lem1}, if we
 take $\kappa$ sufficiently large we have that
 $m_{\sigma}(C)>0$.
 Moreover, since $m_{\sigma  }(U\cap T_{K}^{p}C_{0})=0$,
  $C$ is contained in the complement of $U$, and $T_{K}^{n_{k}-p}C
\subseteq T_{K}^{n_{k}}C_{0}\setminus(T_{K}^{n_{k}}U\triangle U)\subseteq U$ for all $k\in A$ with $k\ge \kappa$ because $C_{0}\subseteq U$. Hence we get that for
 all $k\ge \kappa$, $k\in A$,
$m_{\sigma  }(C\cap T_{K}^{n_{k}-p}C)=0 $. In order to obtain a set $C'$ with positive \mea\   which satisfies $m_{\sigma}(C'\cap
 T_{K}^{n_{k}-p}C')=0$ for all $k\in A$, one has to use again the fact that all the powers of $T_{K}$ are ergodic: if $k_{0}$ is the smallest integer such that $n_{k_{0}}-p\ge 1$, there exists a Borel subset $C_{k_{0}}$ of $C$ with 
$m_{\sigma}(C_{k_{0}})>0$ such that $m_{\sigma}(C_{k_{0}}\cap T_{K}^{n_{k_{0}}-p}C_{k_{0}})=0$, then a Borel subset $C_{k_{0}+1}$ of $C_{k_{0}}$ with 
$m_{\sigma}(C_{k_{0}+1})>0$ such that $m_{\sigma}(C_{k_{0}+1}\cap T_{K}^{n_{k_{0}+1}-p}C_{k_{0}+1})=0$, etc. until we get a Borel subset $C_{\kappa -1}=C'$
 of $C$ such that 
$m_{\sigma}(C')>0$ and $m_{\sigma}(C'\cap T_{K}^{n_{k}-p}C')=0$ for all $k \in A$, $k\ge k_{0}$. 
\par\smallskip
We apply now exactly the same procedure to the \op\ $T_{L}$: if $u_{0}\in H_{L}$ is a non-zero vector, $a',b',c',d'\in\R$ with 
$a'<b'$ and $c'<d'$, applying Lemma \ref{lem3} to the set 
$$V_{a',b'}^{c',d'}=\{u\in H_{L}
\textrm{ ; } \pss{u_{0}}{u}\in R_{a',b'}^{c',d'}\}$$
with $T_{L}$ in place of $T_{K}$ and $m_{\tau}$ in place of $m_{\sigma}$ 
 we obtain a Borel subset
 $D'$
 of $H_{L}$ such that 
$m_{\tau}(D')>0$ and $m_{\tau}(D'\cap T_{L}^{n_{k}-p}D')=0$ for all $k \in B$ with $k\ge k_{0}$.
\par\smallskip
It is now not difficult to see that $C'\times D'$ is a set of positive \mea\ which is \nrect\ \wrt\ the set $\{n_{k}-p \textrm{ ; }
k \ge k_{0}\}\cap \N$ for the \wmx\ \op\ $T_{K}\times T_{L}$ on $H_{K}\times H_{L}$ (indeed, any union of finitely many \nrec\ sets
for \wmx\ transformations is in its turn a \nrec\ set for a \wmx\ system). This finishes the proof of Proposition \ref{prop1}. 
\end{proof}
Theorem \ref{th2} is proved in the case where
$p$ is non-zero.
\subsection{Proof of Theorem \ref{th2} in the case where $p=0$}
The case where $p$ is equal to zero is a bit different, and uses the fact that, since $n_{k}$ divides $n_{k+1}$ for each $k$, $(n_{k})_{k\ge 0}$ is in particular lacunary. Hence there exists a $\lambda _{0}\in \T$ such that for all $k\ge 0$, $|\lambda _{0}^{n_{k}}-1|\ge \delta $. Consider now, as in the proof of Theorem \ref{th1}, the \ops\ $S_{K}=\lambda _{0}T_{K}$ on $H_{K}$ and $S_{L}=\lambda _{0}T_{L}$ on $H_{L}$, where $T_{K}$ and $T_{L}$ are the \ops\ defined above. The same argument as in the proof of Theorem \ref{th1} shows that $S_{K}$ and $S_{L}$ are \wmx\ \wrt\ some \ga\ \mea s on $H_{K}$ and $H_{L}$ respectively, but we need to be a bit more precise here: we need to know that the \mea\ $m_{\sigma  }$ itself is $S_{K}$-\inv, and that $S_{K}$ defines a \wmx\ transformation of $(H_{K}, \mathcal{B}_{K}, m_{\sigma  })$ (and the same thing for the \op\ $S_{L}$ on $H_{L}$). The fact that $m_{\sigma  }$ is $S_{K}$-invariant is an immediate consequence of the fact that \ga\ \mea s are rotation-\inv. In order to check that $S_{K}$ is \wmx\ \wrt\ $m_{\sigma  }$, it suffices to show (see \cite{BayGr1} or \cite{BM} for details) that for all $x,y\in H_{K}$,
$$\frac{1}{N}\sum_{n=1}^{N}\left|\int_{H_{K}}\pss{x}{S_{K}^{n}z}\overline{\pss{y}{z}}dm_{\sigma  }(z)\right|^{2}\to 0 \quad \textrm{as } N\to +\infty .$$
But
\begin{eqnarray*}
 \left|\int_{H_{K}}\pss{x}{S_{K}^{n}z}\overline{\pss{y}{z}}dm_{\sigma  }(z)\right|&=&
 \left|\int_{H_{K}}\lambda _{0}^{n}\pss{x}{T_{K}^{n}z}\overline{\pss{y}{z}}dm_{\sigma  }(z)\right|\\
 &=& \left|\int_{H_{K}}\pss{x}{T_{K}^{n}z}\overline{\pss{y}{z}}dm_{\sigma  }(z)\right|
\end{eqnarray*}
and since $T_{K}$ is \wmx\  the conclusion follows.
\par\smallskip
We are now going to prove the following proposition:

\begin{proposition}\label{prop3}
The set $\{n_{k}\}$
is a \nrec\ set for the \wmx\ system $S_{K}\times S_{L}$ on $(H_{K}\times H_{L}, \mathcal{B}_{K}\times \mathcal{B}_{L}, m_{\sigma  }\times m_{\tau})$.
\end{proposition}

\begin{proof}
We have proved in Lemma \ref{lem3} that given a non-zero vector $x_{0}\in H_{K}$ and $a<b$, $c<d$, the set
 $$U_{a,b}^{c,d}=\{x\in H_{K}
\textrm{ ; } \Re e \pss{x_{0}}{x}\in ]a,b[+i]c,d[\}$$
satisfies $$\sum_{k\in A}m_{\sigma}(U_{a,b}^{c,d}\triangle T_{K}^{n_{k}}U_{a,b}^{c,d})<+\infty.$$
Now take $0<a<b$ and $0<c<d$ such that for each $k\ge 1$, the subsets $R_{a,b}^{c,d}=]a,b[+i]c,d[$ and $\lambda_{0}^{n_{k}}R_{a,b}^{c,d}$
do not intersect. This is possible since $|\lambda_{0}^{n_{k}}-1|\ge \delta$ for each $k\ge 1$. For such a choice of $a,b,c$ and $d$, the series $$\sum_{k\in A}m_{\sigma}(U_{a,b}^{c,d}\cap S_{K}^{-n_{k}}U_{a,b}^{c,d})$$ is convergent.
Indeed, suppose that $x$ belongs to $U_{a,b}^{c,d}\cap S_{K}^{-n_{k}}U_{a,b}^{c,d}$. Then
$\pss{x_{0}}{\lambda_{0}^{n_{k}}T_{K}^{n_{k}}x}\in R_{a,b}^{c,d}$, i.e. $\lambda_{0}^{n_{k}}\pss{x_{0}}{T_{K}^{n_{k}}x} \in R_{a,b}^{c,d}$, i.e. $\pss{x_{0}}{T_{K}^{n_{k}}x} \in \lambda_{0}^{-n_{k}}\, R_{a,b}^{c,d}$. Since $R_{a,b}^{c,d}$ and $\lambda_{0}^{n_{k}}R_{a,b}^{c,d}$
do not intersect, it follows that $\pss{x_{0}}{T_{K}^{n_{k}}x}$ does not belong to $R_{a,b}^{c,d}$, i.e. that $x$ belongs to
$U_{a,b}^{c,d}\setminus T_{K}^{-n_{k}}U_{a,b}^{c,d}$. As $\sum_{k\in A}m_{\sigma}(U_{a,b}^{c,d}\triangle T_{K}^{-n_{k}}U_{a,b}^{c,d})$ is finite, we obtain that
$\sum_{k\in A}m_{\sigma}(U_{a,b}^{c,d}\cap S_{K}^{n_{k}}U_{a,b}^{c,d})$ is finite as well.
Setting
$$C=U_{a,b}^{c,d}\setminus \bigcup_{k\ge \kappa , k\in A}U_{a,b}^{c,d}\cap S_{K}^{n_{k}}U_{a,b}^{c,d},$$ we obtain if $\kappa $ is large enough that
$m_{\sigma  }(C)>0$. It is clear that $C\cap S_{K}^{n_{k}}C=\varnothing$ for all $k\in A$, $k\ge \kappa $. 
The same argument as in Proposition \ref{prop1} above shows that there exists a Borel subset $C'$ of 
$C$ such that $m_{\sigma  }(C')>0$ and $m_{\sigma  }(C'\cap S_{K}^{n_{k}}C')=0$ for all $k\in A$. In a similar fashion we obtain a Borel subset $D'$ of $H_{L}$ with $m_{\tau}(D')>0$ such that $m_{\tau }(D'\cap S_{L}^{n_{k}}D')=0$ for all $k\in B$, and we deduce from this that the set 
$C'\times D'$ is \nrect\ \wrt\ the set $\{n_{k}\}$ for the \wmx\ transformation
$T_{K}\times T_{L}$ of $(H_{K}\times H_{L}, \mathcal{B}_{K}\times \mathcal{B}_{L}, m_{\sigma  }\times m_{\tau})$. This finishes the proof of Proposition \ref{prop3}.
\end{proof}
Theorem \ref{th2} is proved.
We thus obtain a positive answer to Question \ref{q1} in the case where the set $\{n_{k}\}$ is such that $n_{k}|n_{k+1}$ for each $k\ge 0$.

\subsection{Some remarks}
The divisibility assumption on the $n_{k}$'s is used in the proof of Theorem \ref{th1} in two places: first, we need it in order to construct the two \mea s $\sigma  $ and $\tau$ with $|\hat{\sigma  }(n_{k})-1|\le a_{k}$ and $|\hat{\tau  }(n_{k})-1|\le b_{k}$ for each $k\ge 1$, and, second, we deduce from it that the sequence $(n_{k})_{k\ge 0}$ is lacunary. It is not difficult to see that the proof yields in fact the following more general result:

\begin{theorem}\label{cor1}
 Let $(n_{k})_{k\ge 0}$ be  a strictly increasing sequence having the following property: 
 \begin{itemize}
  \item [] for each sequence $(a_{k})_{k\ge 0}$ of positive numbers decreasing to zero, there exists a continuous \mea\ $\sigma  $ on the unit circle such that $|\hat{\sigma  }(n_{k})-1|\le a_{k}$ for all $k$. 
 \end{itemize}
Then for each $p\in\Z\setminus\{0\}$, the set $\{n_{k}-p \textrm{ ; } k\ge 0\}\cap\N$ is a \nrec\ set for some \wmx\ \lds. If moreover $\{n_{k}\}$ is \nrect\ \wrt\ some rotation on $\T$ (in particular if it is lacunary), then the set $\{n_{k}\}$ itself is a \nrec\ set for some \wmx\ \lds.
\end{theorem}

Of course, any sequence  $(n_{k})_{k\ge 0}$  satisfying the assumption ofTheorem \ref{cor1} is a \rs. It is natural to wonder whether it is true that any \rs\ satisfies this assumption. This is not the case, as  shown below.

\begin{example}\label{ex1}
 Let  $(q_{k})_{k\ge 0}$ be a sequence of integers tending to infinity. Consider the sequence  $(n_{k})_{k\ge 0}$  defined by $n_{0}=1$ and $n_{k+1}=q_{k}n_{k}+1$ for all $k\ge 0$. Since $\frac{n_{k+1}}{n_{k}}$ tends to infinity,  $(n_{k})_{k\ge 0}$  is a \rs. But if $\sigma  $ is a continuous \mea\ on $\T$ such that $\hat{\sigma  }(n_{k})$ tends to $1$, then $|\hat{\sigma  }(n_{k})-1|> \frac{1}{q_{k}^{{4}}}$ for infinitely many integers $k$. Indeed, suppose that $|\hat{\sigma  }(n_{k})-1|\le \frac{1}{q_{k}^{{4}}}$ for all $k$'s except finitely many. Then we have for all $k$ sufficiently large
 \begin{eqnarray*}
 \int_{\T}|\lambda -1|d\sigma  (\lambda )&=&\int_{\T}|\lambda ^{n_{k+1}}-\lambda ^{q_{k}n_{k}}|d\sigma  (\lambda )\\
 &\le&\left(\int_{\T}|\lambda ^{n_{k+1}}-1|^{2}d\sigma  (\lambda) \right)^{\frac{1}{2}}+q_{k}\left(\int_{\T}|\lambda ^{n_{k}}-1|^{2}d\sigma  (\lambda )\right)^{\frac{1}{2}}\\
 &\le& \left(2\Re e (1-\hat{\sigma  }(n_{k+1}))\right)^{\frac{1}{2}}+q_{k}\, \left(2\Re e (1-\hat{\sigma  }(n_{k}))\right)^{\frac{1}{2}}\\
 &\le&2\sqrt{2}\, (\frac{1}{q_{k+1}^{2}}+\frac{1}{q_{k}})\cdot
 \end{eqnarray*}
Letting $k$ go to infinity, we obtain that $\hat{\sigma  }(1)=1$, which is impossible since $\sigma  $ is supposed to be continuous. Hence $|\hat{\sigma  }(n_{k})-1|> \frac{1}{q_{k}^{{4}}}$ for infinitely many $k$'s. If $(q_{k})_{k\ge 0}$ goes to infinity extremely slowly, we thus see that any ``rigidity measure'' associated to the sequence $(n_{k})_{k\ge 0}$ has Fourier coefficients going to $1$ along some sub-sequence of the sequence $(n_{k})_{k\ge 0}$ slower than any prescribed rate.
\end{example}

\section{Proof of Theorem \ref{th3}}

Let us begin by recalling briefly here some of the results of \cite{GR}. If $r\ge 1$ is an integer, sets $\{n_{k}^{(r)}\}$  are constructed which have the property of  being \rect\ in the topological sense for all products of $r$ rotations on $\T$. This means for all $(\lambda_1,\ldots, \lambda_{r})\in\T^{r}$ and all $\varepsilon >0$, there exists an integer $k\ge 0$ such that $$\max_{i=1,\ldots, r}|\lambda_{i}^{n_{k}^{(r)}}-1|<\varepsilon.$$ But these sets are \nrect\ for some dynamical system on some compact space, namely for some suitable product of $2^{r-1}+1$ rotations on $\T$: there exist $\mu _{0}, \ldots, \mu _{2^{r-1}}\in\T$ and $\delta >0$ such that for all $k\ge 0$,
$$\min_{j=0,\ldots, 2^{r-1}}|\mu_{j}^{n_{k}^{(r)}}-1|>\delta .$$
The sets $\{n_{k}^{(r)}\}$ have the following form: $\{n_{k}^{(r)}\}=
\{n_{k,0}^{(r)}\}\cup\bigcup_{A\subseteq \{1,\ldots, r-1\}}\{n_{k,A}^{(r)}\}$, with 
$$\{n_{k,0}^{(r)}\}=\bigcup_{N\ge 1}\{H_{N}q+1 \textrm{ ; } 1\le q\le Q_{N}^{(r)}\}:=\bigcup_{N\ge 1}B_{N,0}^{(r)}$$
and
$$\{n_{k,A}^{(r)}\}=\bigcup_{N\ge 1}\{H_{N}\Delta_{N,A}^{(r)}(L_{N}j+1) \textrm{ ; } 1\le j\le \Theta_{N}^{(r)}\}:=\bigcup_{N\ge 1}B_{N,A}^{(r)}$$ for $A\subseteq \{1,\ldots, r-1\}$, $A\not =\varnothing$, and
$$\{n_{k,\varnothing}^{(r)}\}=\bigcup_{N\ge 1}\{H_{N}\Delta_{N,\varnothing}^{(r)}\}:=\bigcup_{N\ge 1}B_{N,\varnothing}^{(r)}.$$
Here $(L_{N})_{N\ge 1}$ is any rapidly increasing sequence of integers, the sequences $(\Delta_{N,A}^{(r)})_{N\ge 1}$, $A\subseteq \{1,\ldots, r-1\}$, $(\Theta _{N}^{(r)})_{N\ge 1}$ and $(Q_{N}^{(r)})_{N\ge 1}$ depend from the sequence $(L_{N})_{N\ge 1}$, and the sequence $(H_{N})_{N\ge 1}$ in an extremely rapidly increasing sequence of integers independent from all the other parameters: none of the sequences $(L_{N})_{N\ge 1}$, $(\Delta_{N,A}^{(r)})_{N\ge 1}$,  $(\Theta _{N}^{(r)})_{N\ge 1}$ or $(Q_{N}^{(r)})_{N\ge 1}$ depend from $(H_{N})_{N\ge 1}$. As explained in \cite{GR}, for a fixed $N\ge 1$ all the blocks 
$B_{N,0}^{(r)}$ and $B_{N,A}^{(r)}$ are necessarily intertwined, but for different $N$'s they are very far away one from another: if at step $N+1$ the integer $H_{N+1}$ is chosen extremely large \wrt\ all the quantities appearing at step $N$, any block
$B_{N+1,0}^{(r)}$ or $B_{N+1,A}^{(r)}$ is very far away from any block $B_{N,0}^{(r)}$ or $B_{N,A}^{(r)}$.
\par\smallskip
Our aim is now to show that if the parameters $L_{N}$ and $H_{N}$ are suitably chosen at each step $N$, the set $\{n_{k}^{(r)}\}$ (which is \rect\ for all products of $r$ rotations) is \nrect\ for some \wmx\ \lds\ on a Hilbert space. As we have seen in Subsections 4.1 and  4.2, any finite union of \nrec\ sets for \wmx\ \lds s is  again a \nrec\ set for some \wmx\ \lds. So it suffices to construct \ops\ $T_{0}$ and $T_{A}$, $A\subseteq \{1,\ldots, r-1\}$ on Hilbert spaces  such that
$\{n_{k,0}^{(r)}\}$ (resp. $\{n_{k,A}^{(r)}\}$) is a \nrec\ set for $T_{0}$ (resp $T_{A}$). But this follows easily from Theorem \ref{th1}.
The first ingredient is the following easy lemma:

\begin{lemma}\label{lem2}
Whatever the choice of $L_{N}$ at step $N$ and the values of $\Delta _{N,A}^{(r)}$ and $\Theta _{N}^{(r)}$, provided that the sequence $(H_{N})_{N\ge 1}$ grows sufficiently fast the sets
$$\bigcup_{N\ge 1}\{H_{N}q \textrm{ ; } 1\le q\le Q_{N}^{(r)}\}$$ and
$$\bigcup_{N\ge 1}\{H_{N}\Delta  _{N,A}^{(r)}L_{N}j \textrm{ ; } 1\le j\le \Theta _{N}^{(r)}\}, \quad A\subseteq \{1,\ldots, r-1\}, \; A\not =\varnothing
$$ 
and
$$\bigcup_{N\ge 1}\{H_{N}\Delta  _{N,\varnothing}^{(r)}\}$$
generate non-\js s.
\end{lemma}

\begin{proof}[Proof of Lemma \ref{lem2}]
By Lemma 2.4 of \cite{GR}, there exists a perfect subset $K$ of $\T$, with $1\in K$, such that for each $\lambda \in K$ and each $N\ge 1$,
$|\lambda ^{H_{N}}-1|\le M\,\frac{H_{N}}{H_{N+1}}$, where $M$ is a numerical constant. Hence 
$$|\lambda ^{H_{N}q}-1|\le M\,\frac{H_{N}Q_{N}^{(r)}}{H_{N+1}}<2^{-N}\quad \textrm{ for all } q \textrm{ with }
1\le q\le Q_{N}^{(r)}
$$ if $H_{N+1}$ is sufficiently large. This holds true for all $\lambda \in K$ and $N\ge 1$. Let now $\varepsilon >0$. If $N_{0}$ is such that $2^{-N_{0}}<\varepsilon $, then $|\lambda ^{H_{N}q}-1|<\varepsilon $ for all $\lambda \in K$, $N\ge N_{0}$ and $1\le q\le Q_{N}^{(r)}$.  Since the set $K$ is perfect and contains the point $1$, one can find $\lambda \in K$ as close to $1$ as we wish, but not equal to $1$. Hence there exists a $\lambda \in K$, $\lambda \not =1$, such that 
$|\lambda ^{H_{N}q}-1|<\varepsilon $ for all $N\ge 1$ and $1\le q\le Q_{N}^{(r)}$. This shows that the set 
$\bigcup_{N\ge 1}\{H_{N}q \textrm{ ; } 1\le q\le Q_{N}^{(r)}\}$ generates a \js. The proof is exactly the same for the other sets, using that $H_{N+1}$ can be chosen much larger than $H_{N}\Delta  _{N,A}^{(r)}L_{N}\Theta _{N}^{(r)}$ for each  subset $A$ of $\{1,\ldots, r-1\}$.
\end{proof}

The second fact we need is that if the sequences  $(L_{N})_{N\ge 1}$
 and $(H_{N})_{N\ge 1}$ grows sufficiently  fast, there exist elements $\lambda _{0}$ and $\lambda _{A}$ of the unit circle, $A\subseteq \{1,\ldots, r-1\}$, such that $$|\lambda _{0}^{n_{k,0}^{(r)}}-1|>\frac{1}{2} \quad \textrm{and}\quad 
 |\lambda _{A}^{n_{k,A}^{(r)}}-1|>\frac{1}{2}$$
 for each $k$. See Proposition 4.5 of \cite{GR} for the proof, which is very similar to the proof of Lemma \ref{lem2} above. Hence we can apply Theorem \ref{th1}, and we obtain that the sets $\{n_{k,0}^{(r)}\}$ and 
 $\{n_{k,A}^{(r)}\}$ are  \nrec\ sets for some \wmx\ \lds s $T_{0}$ and $T_{A}$, respectively, acting on a Hilbert space. This finishes the proof of Theorem \ref{th3}.

\section{Some further \nrect\ lacunary sets}

In this subsection we expand on a result of \cite{BDLR}, where the authors
show that the Chacon transformation is \nrect\ \wrt\ a certain lacunary set. Our Theorem \ref{prop2} is a mild generalization of this, and is inspired by the proof of Proposition 3.10 in \cite{BDLR}.

\begin{theorem}\label{prop2}
 For each strictly increasing sequence $(n_{k})_{k\ge 0}$ of integers, write $n_{k+1}$ as $n_{k+1}=p_{k}n_{k}+r_{k}$, with $p_{k}$ and $r_{k}$ nonnegative integers. If 
 $p_{k}\ge 3$ for each $k$ and if the series $\sum_{k\ge 1}\frac{r_{k}}{p_{k}n_{k}}$ is convergent, then the set $\{n_{k}-1 \textrm{ ; } k\ge 1\}$ is a \nrec\ set for some \wmx\ transformation.
\end{theorem}

We can retrieve from this the examples of \cite{BDLR}:
we can always write $n_{k+1}$ as $n_{k+1}=p_{k}n_{k}+r_{k}$ with $0\le r_{k}<n_{k}$. If the series $\sum_{k\ge 1}\frac{n_{k}}{n_{k+1}}$ is convergent, then $\sum_{k\ge 1}\frac{r_{k}}{p_{k}n_{k}}$ is obviously convergent as well, and so the set $\{n_{k}-1\}$ is a \nrec\ set for some \wmx\ transformation. And if $(n_{k})$ is the sequence
defined by $n_{0}=1$ and $n_{k+1}=3n_{k}+1$ for each $k\ge 0$, then $n_{k}=1+3+\ldots +3^{k}$ for $k\ge 1$. Since the series $\sum_{k\ge 1}\frac{r_{k}}{p_{k}n_{k}}$ is convergent, the set $\{n_{k}-1\}=\{\frac{3^{k+1}-1}{2}-1\}$ is a
\nrec\ set for a certain \wmx\ rank-one transformation (and the proof of Theorem \ref{prop2} shows that this is the Chacon transformation).

\begin{proof}
 Observe first that if there exists a $k_{0}$ such that $r_{k}=0$ for all $k\ge k_{0}$, the conclusion of Theorem \ref{prop2} follows from Theorem \ref{th2}, since in this case $n_{k}$ divides $n_{k+1}$ for each $k\ge k_{0}$. Hence we can suppose without loss of generality that $r_{k}\ge 1$ for infinitely many $k$'s.
 The proof uses the standard cutting and stacking method. Let us denote by $\mathcal{I}_{k}$ the tower of height $h_{k}$ which we have at step $k$ of the construction. At step $k+1$, we cut this tower into $p_{k}$ subtowers $\mathcal{I}_{k,1}, \ldots, \mathcal{I}_{k,p_{k}}$ having a basis which is an interval of length $\frac{1}{p_{k}}$ times the length of the basis of $\mathcal{I}_{k}$. If $r_{k}\ge 1$, set $a_{k}=\lfloor \frac{p_{k}}{3}\rfloor$, and stack on top of each other, and in this order, the towers $\mathcal{I}_{k,1},\ldots, \mathcal{I}_{k,a_{k}}$, one spacer, $\mathcal{I}_{k,a_{k}+1}, \ldots, \mathcal{I}_{k,p_{k}}$, and then $r_{k}-1$ spacers. If $r_{k}=0$, we simply stack the $p_{k}$ towers  $\mathcal{I}_{k,1},\ldots, \mathcal{I}_{k,p_{k}}$ on top of each other.
 So if we start with a tower $\mathcal{I}_{1}$ of height $n_{1}\ge 3$ with a basis which is an interval of length $l_{1}$, the $k^{th}$ tower $\mathcal{I}_{k}$ has height $n_{k}$ and a basis which is an interval of length $l_{k}=\frac{l_{k-1}}{p_{k-1}}$. Since the series $\sum_{k\ge 1}\frac{r_{k}}{p_{k}n_{k}}$ is convergent, with a suitable choice of $l_{1}$ we can ensure that this defines a \mpt\ $T$ of the interval $[0,1]$ (\wrt\ the Lebesgue \mea). It is clear that $T$ is \erg, and the usual argument shows that $T$ is \wmx: suppose that $f\in L^{2}([0,1])$ is an eigenfunction of the Koopman \op\ $U_{T}$, associated to an \eva\ $\lambda \in \T$: $f(Tx)=\lambda f(x)$ a.e. on $[0,1]$, hence for almost every $x\in [0,1]$ we have for each $n\ge 1$ $f(T^{n}x)=\lambda ^{n}f(x)$. As $T$ is ergodic, if $f$ is non-zero we can suppose without loss of generality that $|f|=1$ a.e.. Fix $\varepsilon >0$. Since $r_{k}\ge 1$ for infinitely many $k$'s, we can find an integer $k$ such that $r_{k}\ge 1$ and a function $g\in L^{2}([0,1])$ with the following two properties:  $||f-g||_{2}<\varepsilon $, and $g$ is constant on each level of the tower $\mathcal{I}_{k}$. If  $\tau$ denotes one of the levels of the towers $\mathcal{I}_{k,1},\ldots, \mathcal{I}_{k,a_{k}-1}$ (appearing at the bottom of $\mathcal{I}_{k+1}$), one sees easily that $g$ has the same value on the level  $\tau$ and on the level  $T^{a_{k}n_{k}+1}\tau$. So if $E$ denotes the set which is the union of all these levels  $\tau$, we have
 $$\left(\int_{E} |f(T^{a_{k}n_{k}+1}x)-f(x)|^{2}dx\right)^{\frac{1}{2}}\le 2\varepsilon .$$ Hence
 $$|\lambda ^{a_{k}n_{k}+1}-1|\left(\int_{E}|f(x)|^{2}\right)^{\frac{1}{2}}\le 2\varepsilon .$$ Since $|f|=1$ a.e., we get that  $|\lambda ^{a_{k}n_{k}+1}-1| \sqrt{m(E)}\le 2\varepsilon $. Now, as the measure of $E$ is bigger than $\frac{1}{4}$, this yields that $|\lambda ^{a_{k}n_{k}+1}-1|\le 8\varepsilon $. If $\tau'$ denotes now one of the levels of the towers $\mathcal{I}_{k,a_{k}+1}, \ldots, \mathcal{I}_{k,2a_{k}}$ (appearing just after the first added spacer in the tower $\mathcal{I}_{k+1}$), the same argument shows that $g$ has the same value on the level $\tau'$ and on the level $T^{a_{k}n_{k}}\tau'$, and so if $E'$ denote the union of all these levels $\tau'$ we get that $|\lambda ^{a_{k}n_{k}}-1|\sqrt{m(E')}\le 2\varepsilon $, from which it follows that $|\lambda ^{a_{k}n_{k}}-1|\le 8\varepsilon $. Hence $|\lambda -1|\le 16\varepsilon $ for each $\varepsilon >0$, so $\lambda =1$ and $T$ is \wmx.
 \par\smallskip
 Let us now prove that $\{n_{k}-1\}$  is a \nrec\ set for $T$. Let $A$ denote the first added spacer (in the construction of $\mathcal{I}_{2}$). It is contained in any of the towers $\mathcal{I}_{k}$, $k\ge 3$.
 To visualize the action of $T^{n_{k}-1}$ on $A$, suppose that this spacer $A$ is painted red. Let $\tau $ be a red level of the tower $\mathcal{I}_{k+1}$: if $\tau$ is a level of one of the sub-towers $\mathcal{I}_{k,j}$, $j\in\{1,\ldots, a_{k}-1, a_{k}+1,\ldots, p_{k}-1\}$, then $T^{n_{k}}\tau$ is also a red level, and so $T^{n_{k}-1}\tau$ cannot be red. If $\tau$ is a red level of $\mathcal{I}_{k,a_{k}}$, $T^{n_{k}+1}$ maps $\tau $ on a red level, and the two levels below are not red, so $T^{n_{k}-1}\tau$ cannot be red either. Lastly, we have to consider the case where $\tau$ is a red level of $\mathcal{I}_{k,p_{k}}$: here in general $T^{n_{k}-1}$ does not map $\tau$ onto a level of $\mathcal{I}_{k+1}$, and one has to go over to the towers $\mathcal{I}_{k+2}$, $\mathcal{I}_{k+3}$, etc. in order to see the action of $T^{n_{k}-1}$ on $\tau$. Here is may happen that $T^{n_{k}-1} $ maps some piece of $\tau$ on some red level of a tower $\mathcal{I}_{k+2}$, $\mathcal{I}_{k+3}$, etc. 
 What we eventually get is that for each $k\ge 3$, $T^{n_{k}-1}(A\setminus \mathcal{I}_{k,p_{k}})\cap A=\varnothing$. Set
 $$C=A\setminus \bigcup_{k\ge \kappa }\mathcal{I}_{k,p_{k}},$$ where $\kappa $ is a sufficiently large integer. For each $k\ge 3$, the measure of $\mathcal{I}_{k,p_{k}}$ is less than $\frac{1}{p_{k}n_{k}}$ and, since the series $\sum_{k}\frac{1}{p_{k}n_{k}}$ is convergent, the set $C$ has positive measure if $\kappa $ is large enough. Moreover, $T^{n_{k}-1}C$ is contained in $T^{n_{k}-1}(A\setminus \mathcal{I}_{k,p_{k}})$ for each $k\ge \kappa $, and the set $T^{n_{k}-1}(A\setminus \mathcal{I}_{k,p_{k}})$ does not intersect $A$. Hence it does not intersect $C$ either, and $T^{n_{k}-1}C\cap C=\varnothing$ for all $k\ge \kappa $.
Then the same argument as in the proof of Proposition \ref{prop1} shows that there exists a Borel subset $C'$ of $C$ of positive \mea\ with the property that $m(T^{n_{k}-1}C'\cap C')=0$ for each $k\ge 1$. This proves that $\{n_{k}-1\}$ is a \nrec\ set for $T$.
\end{proof}

As a corollary to Theorem \ref{prop2} we obtain:

\begin{corollary}\label{cor2}
 Let $(n_{k})_{k\ge 0}$ be a  strictly increasing sequence  of integers such that, if we write $n_{k+1}$ as $n_{k+1}=p_{k}n_{k}+r_{k}$, $p_{k}\ge 3$ and $0\le r_{k}<n_{k}$, the series $\sum_{k\ge 1}\frac{r_{k}}{p_{k}n_{k}}$ is convergent.
 Then for each $p\ge 1$
 the set $\{n_{k}-p\}\cap\N$ is a \nrec\ set for some \wmx\ \ds.
\end{corollary}

\begin{proof}
For each $k\ge 1$ we have $n_{k+1}-p+1=p_{k}(n_{k}-p+1)+r_{k}+p_{k}(p-1)$, with $r_{k}+p_{k}(p-1)\ge 0$. Now, since $n_{k+1}\ge 3n_{k}$ for each $k$, the series $\sum \frac{1}{n_{k}}$ is convergent, and it follows that the series
$$\sum_{k\ge 0}\frac{r_{k}+p_{k}(p-1)}{p_{k}(n_{k}-p+1)}$$ is convergent. Applying Theorem \ref{prop2}, we can construct a \wmx\ system for which the set $\{n_{k}-p\}\cap\N$ is \nrect.
\end{proof}

It would be interesting to know whether Theorem \ref{prop2} and Corollary \ref{cor2} remain true when $r_{k}$ is not supposed to be nonnegative, but when one supposes only that the series $\sum_{k\ge 0}\frac{|r_{k}|}{p_{k}n_{k}}$ is convergent.

\end{document}